\documentclass[10pt]{article}
\usepackage{amsmath}
\usepackage{mathtools}
   \mathtoolsset{showonlyrefs,showmanualtags}
\usepackage{amssymb}
\usepackage{amsthm}
\usepackage{lineno}
\usepackage[mathscr]{eucal}
\usepackage{colonequals}
\usepackage{xcolor}
\usepackage[english]{babel}
\usepackage{graphicx}
\theoremstyle{definition}
\newtheorem{definition}{Definition}[section]
\newtheorem{theorem}[definition]{Theorem}
\newtheorem*{theorem*}{Conjecture}

\newtheorem{lemma}[definition]{Lemma}

\theoremstyle{remark}
\newtheorem{remark}[definition]{Remark}
\newtheorem{example}[definition]{Example}
\newcounter{enumctr}

\newcommand{\R}{\mathbb{R}}

\newcommand{\Z}{\mathbb{Z}}

\newcommand{\C}{\mathbb{C}}

\newcommand{\capD}{{}^{\rm{\scriptscriptstyle C\!\!}}D}
\newcommand{\rlD}{{}^{\rm{\scriptscriptstyle R\text{-}L\!\!}}D}

\begin{document}


\title{\vspace*{-10mm}
Stability of scalar nonlinear fractional differential equations with linearly dominated delay}

\author{H.T.~Tuan\footnote{\tt httuan@math.ac.vn, \rm Institute of Mathematics, Vietnam Academy of Science and Technology, 18 Hoang Quoc Viet, 10307 Ha Noi, Viet Nam}
\and
S.~Siegmund\footnote{\tt stefan.siegmund@tu-dresden.de, \rm Center for Dynamics, Faculty of Mathematics, TU Dresden, 01069 Dresden, Germany}}

\maketitle

\begin{abstract}
In this paper, we study the asymptotic behavior of solutions to a scalar fractional delay differential equations around the equilibrium points. More precise, we provide conditions on the coefficients under which a linear fractional delay  equation is asymptotically stable and show that the asymptotic stability of the trivial solution is preserved under a small nonlinear Lipschitz perturbation of the fractional delay differential equation.
\end{abstract}

\section{Introduction}

Let  $A,B\in \R^{d\times d}$ and $f:\R^d\times \R^d\rightarrow \R^d$ be locally Lipschitz continuous.
The existence of solutions of the Caputo fractional differential equation
\begin{equation}\label{eq}
\capD^\alpha_{0+}x(t)=Ax(t)+Bx(t-\tau)+f(x(t),x(t-\tau))
\end{equation}
of order $\alpha\in (0,1)$ with delay $\tau>0$ and continuous initial condition $x(t)=\phi(t)$, $t\in [-\tau,0]$, has been studied in many papers. Abbas \cite{Abbas} used Krasnoselskii's fixed point theorem to show
the existence of at least one local solution. Jalilian and Jalilian \cite{Ji} proved the existence of a global solution on a finite interval by using a fixed point theorem of Leray--Schauder type. Using properties of Mittag-Leffler functions, a weighted norm, and the Banach fixed point theorem, Cong and Tuan \cite{TuanCong1} established the existence and uniqueness of global solutions under a mild Lipschitz condition.

Whenever solutions exist, it is of particular importance
to understand their asymptotic behavior. To the best of our knowledge,
up to now, there have been only very few contributions to the qualitative theory of \eqref{eq}. For $f=0$ and $B=0$, Matignon \cite{Mati} has given a well-known stability criterion based on the spectrum of the matrix $A$. Cermak,  Hornicek and Kisela \cite{Cermak_1} studied the case $f=0$, $A=0$, and obtained a necessary and sufficient condition for the stability of this system. The stability of the system when $A=0$ was discussed by Tuan and Hieu in \cite{TuanHieu}. Regarding the asymptotic behavior of solutions to \eqref{eq} for $f=0$, $d=1$, Stamova \cite{Stamova}, Cermak, Dosla, Kisela \cite{Cermak_2} and He {\em et al.} \cite{He} provided results to characterize the stability of solutions. In the case $f=0$ and $d\geq 1$, Shen and Lam \cite{Shen} considered the stability and performance analysis of the system with the assumptions $A$ is Metzler and $B$ is nonnegative.  Recently, using the properties of Caputo fractional derivatives, the Laplace transform and the Mittag-Leffler function, Thanh, Hieu and Phat \cite{Thanh} proposed sufficient conditions for exponential boundedness, asymptotic
stability and finite-time stability of \eqref{eq} for $f=0$ and $A,B$ arbitrary. However, in contrast to fractional differential equations without delays, the stability theory of delay fractional differential equations \eqref{eq} is far from being fully understood.

In this paper we answer the open question about the relationship
between the stability of the trivial solution of \eqref{eq} and that of its linearization in the scalar case $d=1$. More precise, we consider the scalar delay fractional differential equation
\begin{equation}\label{IntroEq}
\capD_{0+}^\alpha x(t)=ax(t)+bx(t-\tau)+f(x(t),x(t-\tau))
\end{equation}
where $f \colon \R^2 \rightarrow \R$ is locally Lipschitz continuous and satisfies the following conditions:
\begin{itemize}
\item [(H1)] Trivial solution: $f(0,0)=0$,
\item [(H2)] Nonlinearity: $\lim_{\varrho\to 0}\ell_f(\varrho)=0$
with
$$\ell_f(\varrho)
\coloneqq 
\sup_{\substack{x,y,\hat{x},\hat{y}\in B_{\R}(0,\varrho)\\ (x,y)\neq (\hat{x},\hat{y})}}\frac{|f(x,y)-f(\hat{x},\hat{y})|}{\max\{|x-\hat{x}|,|y-\hat{y}|\}}.$$
\end{itemize}
As shown in \cite[Theorem 2.6]{Abbas}, for every continuous initial function $\phi \colon [-\tau, 0] \rightarrow \R$, there exists a unique continuous solution $\varphi(\cdot,\phi) \colon [-\tau, t_{\max}(\phi)) \to \R$ to \eqref{IntroEq} on the maximal interval of existence $[-\tau, t_{\max}(\phi))$ which satisfies the initial condition
\begin{equation}\label{Ini_Cond}
   x(t)=\phi(t),
   \qquad t \in [-\tau,0].
\end{equation}
By (H1), equation \eqref{IntroEq} admits the trivial solution $$\varphi(\cdot,0) \colon [-\tau,\infty) \to \R,\quad t \mapsto 0.$$
For an interval $I \subseteq \R$, let $C(I;\R)$ denote the set of continuous functions $x \colon I \to \R$ with $\|x\|_\infty \coloneqq \sup_{t \in I} |x(t)|$. As in \cite[Definition 1]{TuanHieu}, the trivial solution of \eqref{IntroEq} is called
\begin{align*}
      \text{\emph{stable}} 
   & \ratio\Leftrightarrow
   \forall \varepsilon > 0 \,\exists \delta > 0 \,\forall \phi \in C([-\tau,0];\R) \text{ with } \|\phi\|_{\infty} \leq \delta \colon
\\
   & \qquad
   t_{\max}(\phi) = \infty 
   \text{ and }
   |\varphi(t, \phi)| \leq \varepsilon
   \text{ for } t \in [0,\infty),
\end{align*}
\begin{align*}
      \text{\emph{attractive}} 
   & \ratio\Leftrightarrow
   \exists \delta > 0 \,\forall \phi \in C([-\tau,0];\R) \text{ with } \|\phi\|_{\infty} \leq \delta \colon
\\
   & \qquad
   t_{\max}(\phi) = \infty 
   \text{ and }
   \lim_{t \to \infty} \varphi(t, \phi) =0,
\end{align*}
and
\begin{align*}
    \text{\emph{asymptotically stable}} 
   & \ratio\Leftrightarrow
   \text{the trivial solution is stable and attractive.}
\end{align*}
In Section 5 we provide conditions on $a$, $b$ and $f$ which imply asymptotic stability of the trivial solution of \eqref{IntroEq}. To prepare the proof of this main result, we show a variation of constants formula for \eqref{IntroEq} in Section 2, properties of the characteristic function in Section 3 and estimates for the Mittag-Leffler function in Section 4.

A reader who is familiar with fractional difference equations may skip the remainder of this section, in which we recall notation. 
Let $T > 0$ and $x \colon [0,T] \to \R$ be a measurable function in $L^1([0,T])$, i.e.\ $\int_0^T|x(s)|\;ds<\infty$. Then, the Riemann--Liouville integral of order $\alpha > 0$ is defined by
\[
I_{0+}^{\alpha}x(t):=\frac{1}{\Gamma(\alpha)}\int_0^t(t-s)^{\alpha-1}x(s)\;ds\quad \hbox{ for } t\in (0,T],
\]
where the Gamma function $\Gamma:(0,\infty)\rightarrow \R$ is defined as
\[
\Gamma(\alpha):=\int_0^\infty s^{\alpha-1}\exp(-s)\;ds,
\]
see e.g., Diethelm \cite{Kai}. The corresponding Riemann--Liouville fractional derivative of order $\alpha$ is given by
\[
\rlD_{0+}^\alpha x(t):=(D I_{0+}^{1-\alpha}x)(t) \quad \text{for almost all } t\in (0,T],
\]
where $D=\frac{d}{dt}$ is the usual derivative. The \emph{Caputo fractional derivative} $\capD_{0+}^\alpha x$ of a continuous function $x \colon [0,T] \to \R$ is defined by
\[
\capD_{0+}^\alpha x(t) \coloneqq \rlD_{0+}^\alpha(x(t)-x(0)) \quad \text{for almost all } t\in (0,T].
\]
In a normed space $(X,\| \cdot \|)$ we denote the closed ball with radius $\varrho > 0$ centered at the origin by $B_X(0,\varrho)$.

\section{Variation of constants formula}

In the case $f=0$, the linear initial value problem  \eqref{IntroEq}, \eqref{Ini_Cond}, with continuous initial function $\phi \colon [-\tau,0] \to \R$, has the solution
\[
\varphi(t,\phi)=\phi(0)E^{a,b,\tau}_{\alpha,1}(t)+b\int_{-\tau}^{t-\tau} E^{a,b,\tau}_{\alpha,\alpha}(t-\tau-s)\tilde\phi(s)ds,
\]
where 
\[
E^{a,b,\tau}_{\alpha,\beta}(t):=\mathcal{L}^{-1}\Big(\frac{s^{\alpha-\beta}}{s^\alpha-a-b\exp{(-s\tau)}}\Big)(t),
\]
$\beta=1$ or $\beta=\alpha$, $\mathcal{L}^{-1}$ is the inverse Laplace transform, $\tilde{\phi}$ is the function defined by 
\[
\tilde{\phi}(t)=\begin{cases}
\phi(t),& \text{if}\;t\in [-\tau,0],\\
0,& \text{if}\;t>0.
\end{cases}
\]
If $f$ is globally Lipschitz continuous, using the Laplace transform and the arguments as in \cite[Theorem 3]{Cermak_1}, \cite[Lemma 3.1]{TuanCong}, and \cite[Lemma 1]{TuanHieu}, we obtain the following variation of constants formula for \eqref{IntroEq}.

\begin{lemma}[Variation of constants formula for delay fractional differential equations]
Assume that $f:\R\times \R\to \R$ is Lipschitz continuous and $\phi:[-\tau,0]\to \R$ is continuous. Then, equation \eqref{IntroEq} with the initial condition \eqref{Ini_Cond} has a unique solution $\varphi(\cdot,\phi)$ on $[-\tau,\infty)$. Moreover, this solution satisfies
\begin{align}\label{var_const_for}
\notag\varphi(t,\phi)&=\phi(0)E^{a,b,\tau}_{\alpha,1}(t)+b\int_{-\tau}^{t-\tau} E^{a,b,\tau}_{\alpha,\alpha}(t-\tau-s)\tilde\phi(\tau)d\tau\\
&\hspace{0.5cm}+\int_0^t E^{a,b,\tau}_{\alpha,\alpha}(t-s)f(\varphi(s),\varphi(s-\tau))ds \quad \text{for } t > 0.
\end{align} 
\end{lemma}
\begin{proof}
From \cite[Corollary 3.2]{TuanCong1}, we see that for any continuous initial data $\phi$, equation \eqref{IntroEq} with the initial condition $x(t)=\phi(t)$ on $[-\tau,0]$ has a unique solution on $[-\tau,\infty)$. Moreover, this solution is exponentially bounded, see \cite[Theorem 4.1]{TuanCong1}. Taking the Laplace transform on both sides of \eqref{IntroEq} and using the facts that
\begin{align*}
\mathcal{L}(\capD^\alpha_{0+}x(t))(s)&=s^\alpha \mathcal{L}(x(t))(s)-s^{\alpha-1}x(0)=s^\alpha X(s)-s^{\alpha-1}\phi(0),\\
\mathcal{L}(x(t-\tau)(s)&=\exp{(-\tau s)}\mathcal{L}(x(t))(s)+\exp{(-\tau s)}\int_{-\tau}^0 \exp{(-s u)}x(u)\;du\\
&=\exp{(-\tau s)}X(s)+\exp{(-\tau s)}\int_{-\tau}^0 \exp{(-s u)}\phi(u)\;du,
\end{align*}
for $s\in \{z\in \C:\Re z>c\}$, $c$ large enough, we get
\begin{align}\label{v_1}
\notag X(s)&=\frac{s^{\alpha-1}\phi(0)}{s^\alpha-a-b\exp{(-\tau s)}}+\frac{b\exp{(-\tau s)}\int_{-\tau}^0 \exp{(-su)}\phi(u)\;du}{s^\alpha-a-b\exp{(-\tau s)}}\\
&\quad {}+\frac{F(s)}{s^\alpha-a-b\exp{(-\tau s)}},
\end{align}
where $X(s)\coloneqq \mathcal{L}(x(t))(s)$ and $F(s)=\mathcal{L}(f(x(t),x(t-\tau)))(s)$. Applying the inverse Laplace transform on both sides of \eqref{v_1}, we obtain
\begin{align}\label{v_2}
\notag x(t)&=E^{a,b,\tau}_{\alpha,1}(t)\phi(0)+b\int_{-\tau}^{t-\tau} E^{a,b\,\tau}_{\alpha,\alpha}(t-\tau-s)\tilde\phi(s)\;ds\\
&\quad{}+\int_0^t E^{a,b,\tau}_{\alpha,\alpha}(t-s)f(x(s),x(s-\tau))\;ds \quad \text{for } t > 0.
\end{align}
Here, to obtain \eqref{v_2}, we used
\begin{align*}
\mathcal{L}^{-1}\left( \frac{F(s)}{s^\alpha-a-b\exp{(-\tau s)}}\right)(t)&=\mathcal{L}^{-1}\left( F(s)\mathcal{L}(E^{a,b,\tau}_{\alpha,\alpha}(t))(s)\right)(t)\\
&=\mathcal{L}^{-1}\left(\mathcal{L}(f(x(\cdot),x(\cdot-\tau))\ast E^{a,b,\tau}_{\alpha,\alpha}(\cdot)(t))(s)\right)(t)\\
&=\int_0^t E^{a,b,\tau}_{\alpha,\alpha}(t-s)f(x(s),x(s-\tau))\;ds,
\end{align*}
and
\begin{align*}
&\mathcal{L}^{-1}\left(\frac{\exp{(-\tau s)}\int_{-\tau}^0 \exp{(-su)}\phi(u)\;du}{s^\alpha-a-b\exp{(-\tau s)}}\right)(t)\\
&\hspace{2cm}=\mathcal{L}^{-1}\left(\mathcal{L}(E^{a,b\,\tau}_{\alpha,\alpha}(t))(s)\mathcal{L}(\tilde{\phi}(t-\tau))(s)\right)(t)\\
&\hspace{2cm}=\mathcal{L}^{-1}\left(\mathcal{L}(E^{a,b\,\tau}_{\alpha,\alpha}(\cdot)\ast \tilde{\phi}(\cdot-\tau)(t) )(s)\right)(t)\\
&\hspace{2cm}=\int_{-\tau}^{t-\tau} E^{a,b\,\tau}_{\alpha,\alpha}(t-\tau-s)\tilde\phi(s)\;ds,
\end{align*}
where $\ast$ denotes the convolution operator.
\end{proof}

\section{Properties of the characteristic function}

To derive the asymptotic behavior of the solutions to \eqref{IntroEq} from \eqref{var_const_for}, we need to study the function $E^{a,b,\tau}_{\alpha,\beta}(t)$. First, we recall some facts concerning the zeros of the characteristic function $Q(s):=s^\alpha-a-b\exp{(-s\tau)}$.
\begin{lemma}\label{zero_distr}
Let $\alpha\in (0,1)$, $a,b\in \R$, $\tau>0$. Then the following statements hold.
\begin{itemize}
\item[(i)] If $a+b\geq 0$, then the equation $Q(s)=0$ has at least one nonnegative real root.
\item[(ii)] If $s$ is a zero of $Q$, then its complex conjugate $\bar{s}$ also satisfies $Q(\bar{s})=0$.
\item[(iii)] Let $0<\omega<\pi$. Then the equation $Q(s)=0$ has at most finitely many roots $s$ such that $|\arg(s)|\leq \omega$.
\item[(iv)] The equation $Q(s)=0$ has no more than a finite number of roots in any vertical strip of the complex plane given by
\[
\{z\in \C \colon \rho_1\leq \Re{(z)}\leq \rho_2\}.
\]
\item[(v)] There does not exist an $s\in \C\setminus\{0\}$ satisfying $Q(s)=Q'(s)=Q''(s)=Q'''(s)=0$. 
\end{itemize}
\end{lemma}
\begin{proof}
For the proof of $\textup{(i)}$--$\textup{(iii)}$, see \cite[Proposition 2]{Cermak_2}. 

\noindent (iv) Assume that $s=x+iy\in \{z\in \C \colon \rho_1\leq \Re{(z)}\leq \rho_2\}$. Choosing $T_0>0$ such that $T^\alpha_0>|a|+|b|\exp{(-\rho_1\tau)}$. Then, for any $s=x+iy\in \{z\in \C \colon \rho_1\leq \Re{(z)}\leq \rho_2\}\cap \{z\in\C \colon |\Im{(z)}|\geq T_0\}$, we have 
\[
|s^\alpha|>|a|+|b|\;|\exp{(-s\tau)}|,
\]
which implies that the equation $Q(s)=0$ has no solution in the set $\{z\in \C \colon \rho_1\leq \Re{(z)}\leq \rho_2\}\cap \{z\in\C \colon |\Im(z)|\geq T_0\}$. On the other hand, the function $Q(s)$ has only finitely many roots in the compact set $\{z=x+iy\in \C \colon \rho_1\leq x\leq \rho_2,\;-T_0\leq y\leq T_0\}$ (the function $Q$ is analytic in this domain). Hence, there exist at most finitely many roots of $Q(s)$ in $\{z\in \C \colon \rho_1\leq \Re{(z)}\leq \rho_2\}.$ 

\noindent (v) Now we assume that there exists $s\in \C\setminus\{0\}$ such that $Q(s)=Q'(s)=Q''(s)=Q'''(s)=0$. Using $Q'(s)=Q''(s)=0$, a direct computation shows
\begin{equation*}
b\exp{(-\tau s)}=-\frac{\alpha s^{\alpha-1}}{\tau}=\frac{\alpha (\alpha-1)s^{\alpha-2}}{\tau^2},
\end{equation*}
which implies 
\begin{equation}\label{stp_1}
s=\frac{1-\alpha}{\tau}.
\end{equation}
Similarly, from the equality $Q''(s)=Q'''(s)=0$, we have
\begin{equation*}
b\exp{(-\tau s)}=\frac{\alpha (\alpha-1)s^{\alpha-2}}{\tau^2}=-\frac{\alpha(\alpha-1)(\alpha-2)s^{\alpha-3}}{\tau^3},
\end{equation*}
which implies 
\begin{equation*}
s=\frac{2-\alpha}{\tau},
\end{equation*}
a contradiction to \eqref{stp_1}. The proof of $\textup{(iv)}$ completes.
\end{proof}
The following lemma provides a condition which ensures that all solutions $s \in \C$ of $Q(s)=0$ satisfy $\Re (s) < 0$. It is stated without proof in \cite[Proposition 4]{Cermak_2}, we give a simple and geometric proof for completeness.
\begin{lemma}\label{key_lemma}
Let $\alpha\in (0,1)$, $a,b\in \R$ and $\tau>0$. If $a\leq b<-a$, then the equation $Q(s)=0$ has no root with non-negative real part.
\end{lemma}
\begin{proof}
Define $C_+\coloneqq\{z\in \C:\Re (z)\geq 0\}$ and the functions $w_1\colon C_+\to \C$, $w_1(s)=a+b\exp{(-\tau s)}$ and $w_2 \colon C_+\to \C$, $w_2(s)=s^\alpha$. It is obvious that $D_2\coloneqq w_2(C_+)=\{s\in \C\colon|\arg(s)|\leq \frac{\alpha \pi}{2}\}$. On the other hand, for any $s\in C_+$,
\[
|w_1(s)-a|\leq |b|.
\]
Hence, $D_1\coloneqq w_1(C_+)=\{s\in \C \colon |s-a|\leq |b|\}$. This shows that if $a<0$ and $|b|<|a|$ then $D_1\cap D_2=\emptyset$ (see Figure 1), that is, there does not exist $s\in C_+$ such that $w_1(s)=w_2(s)$. Now we consider $a<0$ and $a=b$. In this case $D_1 \cap  D_2=\{0\}$. Assume that there is a $s\in C_+$ such that $w_1(s)=w_2(s)$. Then $w_2(s)=0$ which implies $s=0$. However, $w_1(0)\ne 0$, a contradiction. Combining the arguments as above, we conclude that if $a\leq b<|a|$ then the equation $Q(s)=0$ has no solution with non-negative real part.
\end{proof}
\begin{figure}
\centering
`\includegraphics[width=0.85\textwidth]{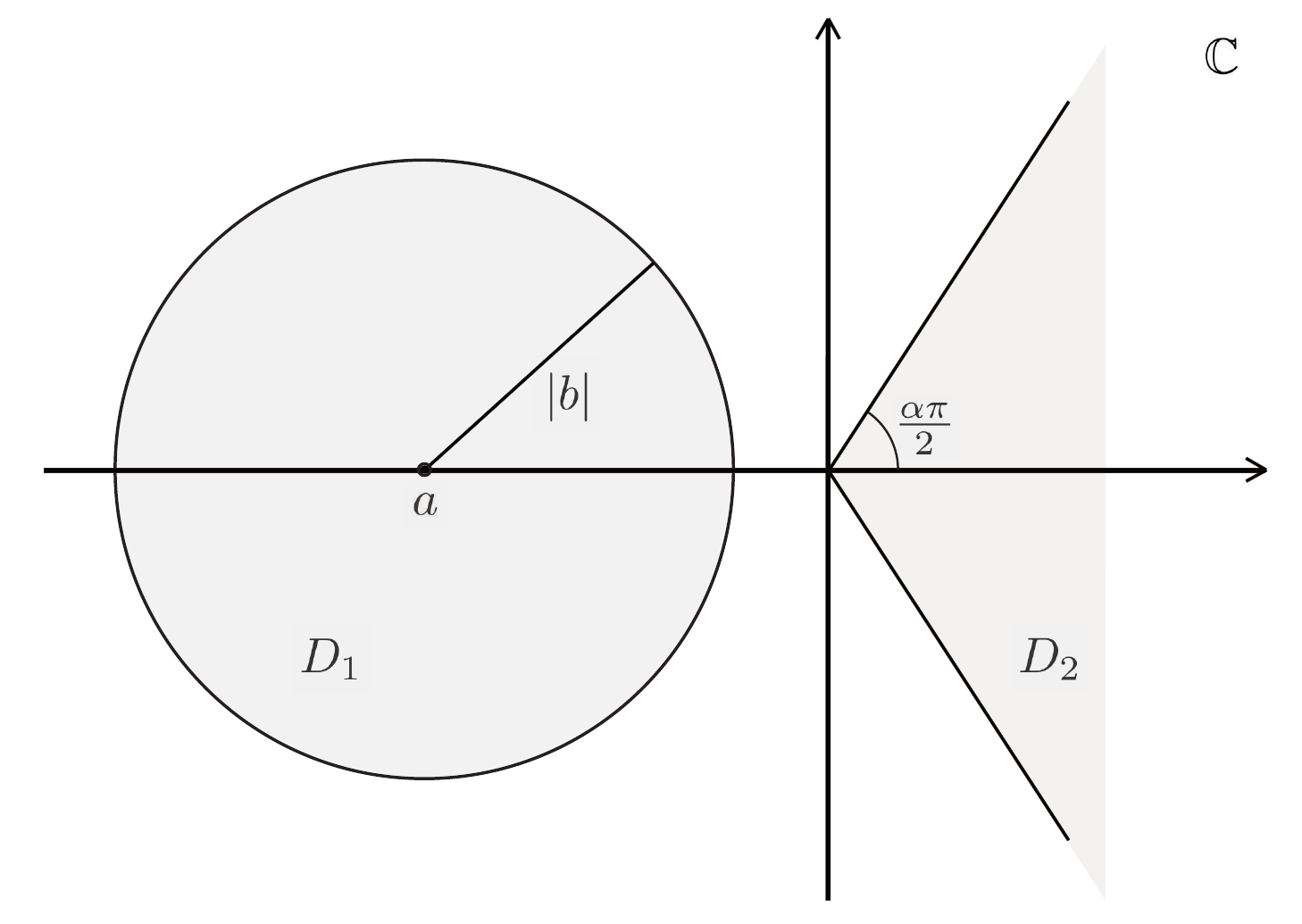}
\caption{The domains $D_1$ and $D_2$ in the case $a<0,\;|b|<|a|$.}
\end{figure}

\section{Asymptotics of Mittag-Leffler functions}

In the following lemma, we provide some estimates involving Mittag-Leffler functions $E^{a,b,\tau}_{\alpha,\beta}(t)$ under assumptions which ensure that all roots of the equation $Q(s)=0$ have negative real parts.
\begin{lemma}\label{est_lemma}
Let $\alpha\in (0,1)$, $\tau>0$ and $a,b\in \R$ satisfying $a\leq b<-a$. Then, there exists a constant $C>0$ such that the following estimates hold:
\begin{itemize}
\item[(i)] $|E^{a,b,\tau}_{\alpha,1}(t)|\leq \frac{C}{t^{\alpha}}$ for all $t\geq 1$.
\item[(ii)] $|E^{a,b,\tau}_{\alpha,\alpha}(t)|\leq \frac{C}{t^{\alpha+1}}$ for all $t\geq 1$.
\item[(iii)] $\int_0^\infty |E^{a,b,\tau}_{\alpha,\alpha}(s)|ds\leq C$.
\end{itemize}
\end{lemma}
\begin{proof}
In the case $b=0$, the function $E^{a,0,\tau}_{\alpha,\beta}(t)$ equals $t^{\beta-1}E_{\alpha,\beta}(at^\alpha)$ and this lemma is proved in \cite[Theorems 2 \& 3]{TuanCongSon}. Hence, we only discuss the remaining case $a\leq b<-a,\; b\ne 0$.

We define for $\mu>0$ and $\theta\in (0,\pi)$ an oriented contour $\gamma(\mu,\theta)$ formed by three segments:
\begin{itemize}
\item $\{s \in \C \colon \arg{(s)}=-\theta, |s|\geq \mu\}$,
\item $\{s \in \C \colon -\theta\leq \arg{(s)}\leq \theta, |s|=\mu\}$,
\item $\{s \in \C \colon \arg{(s)}=\theta, |s|\geq \mu\}$,
\end{itemize}
\begin{figure}
\centering
\includegraphics[width=0.70\textwidth]{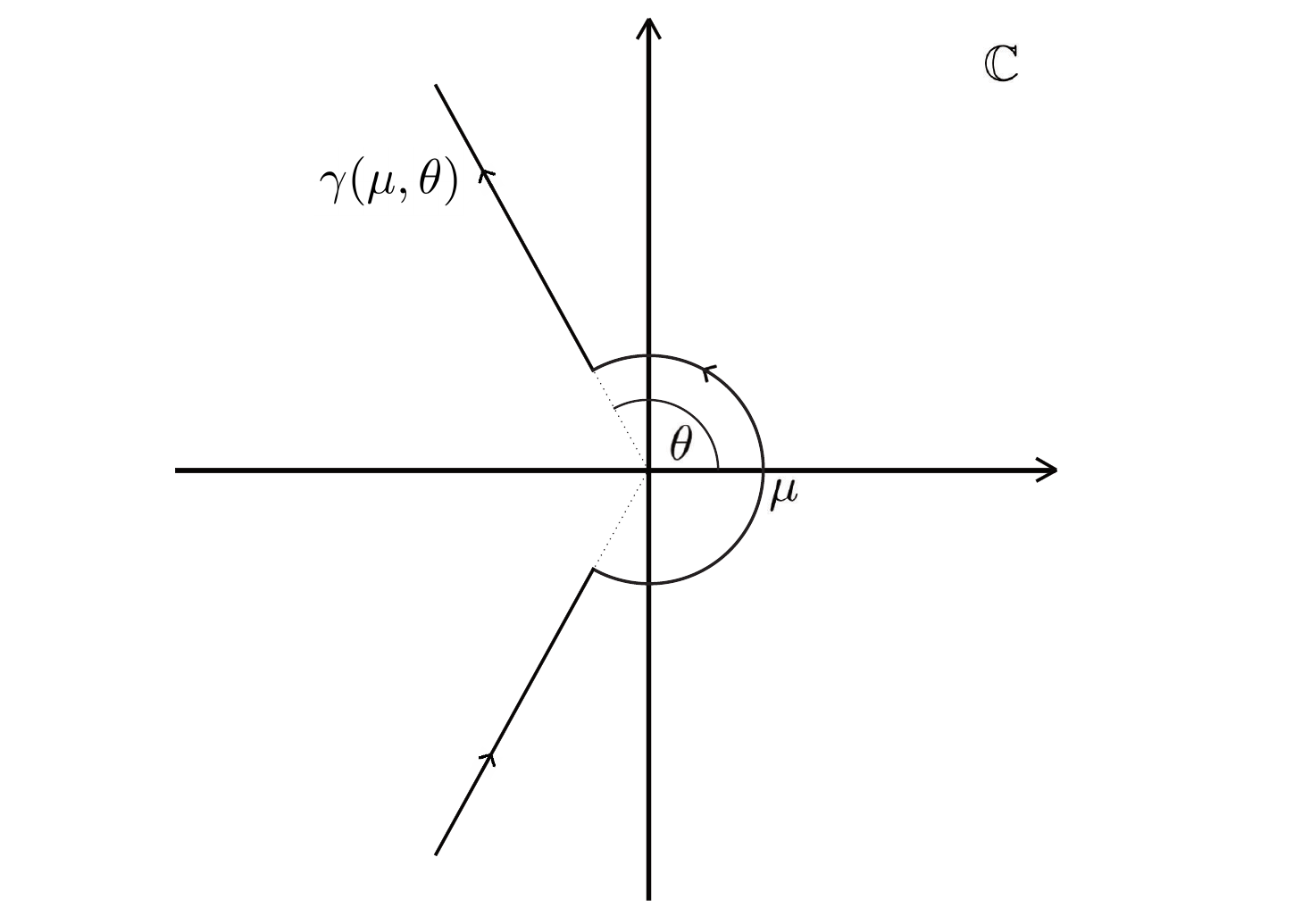}
\caption{The contour $\gamma(\mu,\theta)$.}
\end{figure}
see Figure 2. By Lemma \ref{zero_distr}(iii), there exists a $\delta>0$ such that the function $s^\alpha-a-b \exp{(-\tau s)}$ has no zeros $s_i$ with
$|\arg{(s_i)}|= \frac{\pi}{2}+\delta$, and there are only finitely many zeros $s_i$ which satisfy $|\arg{(s_i)}|\leq \frac{\pi}{2}+\delta$. Hence, there exist $R>0$ such that all zeros $s_i$ lie to the left of $\gamma(R,\frac{\pi}{2}+\delta)$. Due to the fact that $0$ is not a root of $Q(s)=0$, we can find $0<\varepsilon<R$ such that $Q(s)=0$ has no solutions inside and on the circle $\{z\in \C:|z|=\varepsilon\}$. For $t\geq 1$, from \cite[p.\ 346]{Cermak_2}, we have
\begin{align*}
E_{\alpha,\beta}^{a,b,\tau}(t)&=\frac{1}{2\pi i}\int_{\gamma(R,\frac{\pi}{2}+\delta)}\frac{s^{\alpha-\beta}\exp{(ts)}}{s^\alpha-a- b\exp{(-\tau s)}}\;ds\\
&=I^1(t)+I^2(t),
\end{align*}
where
\[
I^1(t)=\frac{1}{2\pi i}\int_{\gamma(\frac{\varepsilon}{t},\frac{\pi}{2}+\delta)}\frac{s^{\alpha-\beta}\exp{(ts)}}{s^\alpha-a-b \exp{(-\tau s)}}\;ds,
\]
and
\[
I^2(t)=\frac{1}{2\pi i}\int_{\substack{\gamma(R,\frac{\pi}{2}+\delta)-\gamma(\frac{\varepsilon}{t},\frac{\pi}{2}+\delta)}}\frac{s^{\alpha-\beta}\exp{(ts)}}{s^\alpha-a-b\exp{(-\tau s)}}\;ds.
\]
If there are no solutions of $Q(s)=0$ in the domain bounded by $\gamma(R,\frac{\pi}{2}+\delta)-\gamma(\varepsilon,\frac{\pi}{2}+\delta)$, then 
\begin{equation}\label{i_1}
I^2(t)=0 \quad \text{for all } t\geq 1.
\end{equation}
Now assume that the roots of $Q(s)=0$ in the domain bounded by $\gamma(R,\frac{\pi}{2}+\delta)-\gamma(\varepsilon,\frac{\pi}{2}+\delta)$ are $s_1,\dots,s_N$. Using the Cauchy residue theorem (see e.g., \cite[Theorem 6.16, p.\ 347]{Zill}), we have
\[
I^2(t)=\sum_{i=1}^N {\textup{Res}}_{s=s_i}\left(\frac{s^{\alpha-\beta}\exp{(st)}}{s^{\alpha}-a-b\exp{(-\tau s)}}\right)
\quad\text{for all } t \geq 1.
\]
From the proof of \cite[Lemma 2]{Cermak_2} we deduce that
\begin{equation}\label{i_2}
I^2(t)=\sum_{i=1}^N \left(a^i_{-1}+a^i_{-2} t+\frac{a^i_{-3}t^2}{2}\right)\exp{(s_it)}
\quad\text{for all } t \geq 1,
\end{equation}
where $N$, $a^i_{-1},a^i_{-2},a^i_{-3}$, $i=1,\dots,N$, are independent of $t$. Note that $\Re (s_i) < 0$ for $i=1,\dots,N$ and the function $\exp{(s_i t)}$ converges exponentially to $0$ as $t \to \infty$. 

\noindent (i) Let $\beta=1$. For the term $I^1(t)$, by the change of variables $s=\frac{u^{1/\alpha}}{t}$, we have
\[
I^1(t)=\frac{t^{\beta-1}}{2\alpha \pi i}\int_{\gamma(\varepsilon^\alpha,\frac{\alpha \pi}{2}+\alpha \delta)}\frac{u^{\frac{1-\beta}{\alpha}}\exp{[(1+\tau/t)u^{1/\alpha}]}}{(u-at^\alpha)\exp{(\frac{\tau u^{1/\alpha}}{t})}-bt^\alpha}du\quad \text{for all } t\geq 1.
\]
Set $D\coloneqq \{z\in \C \colon |z|\leq \varepsilon,\; |\arg(z)|\leq \pi/2+\delta\}$ and $\nu \coloneqq \min_{s\in \gamma(\varepsilon,\pi/2+\delta)\cup D}|s^{\alpha}-a-b\exp{(-\tau s)}|$. Due to $|s^{\alpha}-a-b\exp{(-\tau s)}|\geq \nu$ for all $s\in \gamma(\varepsilon/t,\pi/2+\delta)$ and $t\geq 1$, we have
\begin{equation*}
|(u-at^\alpha)\exp{(\tfrac{\tau u^{1/\alpha}}{t})}-bt^\alpha|\geq \nu t^\alpha |\exp{(\tfrac{\tau u^{1/\alpha}}{t})}|
\end{equation*}
for all $u\in \gamma(\varepsilon^\alpha, \frac{\alpha \pi}{2}+\alpha \delta)$ and $t\geq 1$. This implies that
\[
|I^1(t)|\leq \frac{1}{2\alpha \pi \nu}\int_{\gamma(\varepsilon^\alpha, \frac{\alpha \pi}{2}+\alpha \delta)}|\exp{(u^{1/\alpha})}|\;|du|\; \frac{1}{t^\alpha}
\]
for all $t\geq 1$, which together with \eqref{i_1} and \eqref{i_2} completes the proof of this part.

\noindent (ii) Consider $\beta=\alpha$. For all $t\geq 1$, we have
\begin{align}
\notag I^1(t)&=-\frac{1}{b}\frac{1}{2\pi i}\int_{\gamma(\frac{\varepsilon}{t},\frac{\pi}{2}+\delta)}\exp{((t+\tau)s)}\;ds\\
\notag &\quad{}+\frac{1}{b}\frac{1}{2\pi i}\int_{\gamma(\frac{\varepsilon}{t},\frac{\pi}{2}+\delta)}\frac{(s^\alpha-a)\exp{((t+\tau)s)}}{s^\alpha-a-b\exp{(-\tau s)}}\;ds\\
&=-\frac{1}{b}J_1(t)+\frac{1}{b}J_2(t)-\frac{a}{b}J_3(t),\label{ii_0}
\end{align}
where, using \cite[Formula (1.52), p.\ 16]{Podlubny},
\begin{align}
\notag J_1(t)&=\frac{1}{2\pi i}\int_{\gamma(\frac{\varepsilon}{t},\frac{\pi}{2}+\delta)}\exp{((t+\tau)s)}\;ds\\
\notag &=\frac{1}{2\alpha \pi i}\int_{\gamma((1+\tau/t)^\alpha\varepsilon^\alpha,\frac{\alpha \pi}{2}+\alpha \delta)}u^{(1-\alpha)/\alpha}\exp{(u^{1/\alpha})}\;du\;\frac{1}{t+\tau}\\
\notag&=\left(\frac{1}{\Gamma(z)}\right)_{|z=0}\;\frac{1}{t+\tau}\\
&=0,\label{ii_1}
\end{align}
\begin{align}
\notag|J_2(t)|&=\left|\frac{1}{2\pi i}\int_{\gamma(\frac{\varepsilon}{t},\frac{\pi}{2}+\delta)}\frac{s^\alpha\exp{((t+\tau)s)}}{s^\alpha-a-b\exp{(-\tau s)}}\;ds\right|\\
&\leq \frac{1}{\nu_1 \alpha \pi}\int_{0}^\infty r^{1/\alpha}\exp{(-r^{1/\alpha}\sin\delta)}\;dr\;\frac{1}{t^{\alpha+1}}\label{ii_2}
\end{align}
with $\nu_1=\inf_{s\in \gamma(0,\frac{\pi}{2}+\delta)}|s^\alpha-a-b\exp{(-\tau s)}|$,
and
\begin{align*}
J_3(t)&=I^1(t)+\frac{1}{2\pi i}\int_{\gamma(\frac{\varepsilon}{t},\frac{\pi}{2}+\delta)}\frac{\exp{(ts)}[\exp{(\tau s)}-1]}{s^\alpha-a-b\exp{(-\tau s)}}\;ds\\
&=I^1(t)+G(t).
\end{align*}
Note that
\begin{align}
\notag |G(t)|&=\left| \frac{1}{2\pi i}\int_{\gamma(0,\frac{\pi}{2}+\delta)}\frac{\exp{(ts)}[\exp{(\tau s)}-1]}{s^\alpha-a-b\exp{(-\tau s)}}\;ds\right|\\
\notag &=\left|\frac{1}{2\alpha\pi i}\int_{\gamma(0,\frac{\alpha\pi}{2}+\alpha\delta)}\frac{\exp{(u^{1/\alpha})}[\exp{(\tau u^{1/\alpha}/t)}-1]}{u/{t^\alpha}-a-b\exp{(-\tau u^{1/\alpha}/t)}}u^{1/{\alpha}-1}\;ds\;\frac{1}{t}\right|\\
&\leq \frac{\tau}{\nu_1 \alpha \pi}\int_0^\infty r^{(2-\alpha)/\alpha}\exp{(-r^{1/\alpha}\sin \delta)}\;dr\;\frac{1}{t^2},\label{ii_3}
\end{align}
for all $t\geq 1$, where  for $z\in \gamma(0,\frac{\pi}{2}+\delta)$, we used the inequality
\[
|\exp{(z)}-1|\leq |z|.
\]
On the other hand, from \eqref{ii_0} we have
\[
I^1(t)=-\frac{1}{a+b}J_1(t)+\frac{1}{a+b}J_2-\frac{a}{a+b}G(t),
\]
which together with \eqref{ii_1}, \eqref{ii_2} and \eqref{ii_3} shows that
\begin{align*}
|I^1(t)|&\leq \frac{1}{|a+b|\nu_1 \alpha \pi}\int_0^\infty r^{1/\alpha}\exp{(-r^{1/\alpha}\sin\delta)}dr\;\frac{1}{t^{\alpha+1}}\\
&\quad{}+\frac{|a|\tau}{|a+b|\nu_1\alpha \pi}\int_0^\infty r^{(2-\alpha)/\alpha}\exp{(-r^{1/\alpha}\sin\delta)}dr\;\frac{1}{t^2}
\end{align*}
for $t\geq 1$. This combines with \eqref{i_1} and \eqref{i_2} to complete the proof of this part.

\noindent (iii) First we consider $t\in [0,1]$. For $R>0$ and $\delta>0$ chosen as above, we split the contour $\gamma(R,\frac{\pi}{2}+\delta)$ into three parts: $\gamma(R,\frac{\pi}{2}+\delta)=\gamma_1(R,\frac{\pi}{2}+\delta)\cup\gamma_2(R,\frac{\pi}{2}+\delta)\cup\gamma_3(R,\frac{\pi}{2}+\delta)$, where
\begin{align*}
   \gamma_1(R,\tfrac{\pi}{2}+\delta)
   & \coloneqq
   \left\{r(\cos\varphi+i\sin\varphi) \in \C \colon R\leq r<\infty,\; \varphi=\tfrac{\pi}{2}+\delta\right\},
\\
   \gamma_3(R,\frac{\pi}{2}+\delta)
   & \coloneqq
   \left\{r(\cos\varphi+i\sin\varphi) \in \C \colon R\leq r\leq \infty,\; \varphi=-(\tfrac{\pi}{2}+\delta)\right\},
\\ \intertext{and}
   \gamma_2(R,\tfrac{\pi}{2}+\delta)
   & \coloneqq
   \left\{R(\cos\varphi+i\sin\varphi) \in \C \colon -(\tfrac{\pi}{2}+\delta)\leq \varphi\leq \tfrac{\pi}{2}+\delta\right\}.
\end{align*}

Taking $R_1>R$ such that $|s^\alpha-a-b\exp{(-\tau s)}|>\frac{|b\exp{(-\tau s)}|}{2}$ for all $s=r\exp{(i\varphi)}\in \gamma_i(R,\frac{\pi}{2}+\delta),\; r\geq R_1$, $i=1$ or $i=3$. On $\gamma_i(R,\frac{\pi}{2}+\delta)$, we obtain the estimate
\begin{align}\label{e_1}
&\notag \quad\left|\int_{\gamma_i(R,\frac{\pi}{2}+\delta)}\frac{\exp{(st)}}{s^\alpha-a-b\exp{(-\tau s)}}ds\right|\\
\notag&\leq \left|\int_{\{s=r\exp{(i\varphi)}\in\gamma_i(R,\frac{\pi}{2}+\delta):R\leq r\leq R_1\}}\frac{\exp{(st)}}{s^\alpha-a-b\exp{(-\tau s)}}ds\right|\\
\notag&\quad{}+\left|\int_{\{s=r\exp{(i\varphi)}\in\gamma_i(R,\frac{\pi}{2}+\delta):r\geq R_1\}}\frac{\exp{(st)}}{s^\alpha-a-b\exp{(-\tau s)}}ds\right|\\
&\leq \frac{R_1-R}{\eta}+\frac{2}{|b|\tau R_1\sin\delta}
\end{align}
for all $t\in [0,1]$, and $i=1$ or $i=3$, where $\eta:=\inf_{s\in \gamma(R,\frac{\pi}{2}+\delta)}\{|s^\alpha-a-b\exp{(-\tau s)}|\}$. Moreover, on $\gamma_2(R,\frac{\pi}{2}+\delta)$, we have
\begin{equation}\label{e_2}
\left|\int_{\gamma_2(R,\frac{\pi}{2}+\delta)}\frac{\exp{(ts)}}{s^\alpha-a-b\exp{(-\tau s)}}ds\right|\leq \frac{2\pi R \exp{(R)}}{\eta},
\end{equation}
for all $t\in [0,1]$. Combining \eqref{e_1} and \eqref{e_2} leads to the estimate
\begin{equation}\label{e_3}
\int_0^t |E_{\alpha,\alpha}^{a,b,\tau}(s)|ds\leq \frac{R_1-R}{\pi \eta}+\frac{2}{\pi|b|\tau R_1\sin\delta}+\frac{ R \exp{(R)}}{\eta}
\end{equation}
for all $t\in [0,1]$.

To complete the proof of this part, we will show that the statement also holds for $t>1$. Using \eqref{e_3} and Lemma \ref{est_lemma}$\textup{(ii)}$, there exists a constant $C_3>0$ such that the following estimate holds
\begin{align}
\notag \int_0^t |E_{\alpha,\alpha}^{a,b,\tau}(s)|ds&\leq \int_0^1 |E_{\alpha,\alpha}^{a,b,\tau}(s)|ds+\int_1^t |E_{\alpha,\alpha}^{a,b,\tau}(s)|ds\\
\notag&\leq \frac{R_1-R}{\pi\eta}+\frac{2}{\pi|b|\tau R_1\sin\delta}+\frac{ R \exp{(R)}}{\eta}+\int_1^t\frac{C_3}{s^{\alpha+1}}ds\\
\notag&\leq \frac{R_1-R}{\pi\eta}+\frac{2}{\pi|b|\tau R_1\sin\delta}+\frac{ R \exp{(R)}}{\eta}+\frac{C_3}{\alpha}.
\end{align}
Thus, there exists $C>0$ such that
\[
   \int_0^\infty |E^{a,b,\tau}_{\alpha,\alpha}(s)|ds=\sup_{t\geq 0}\int_0^t |E^{a,b,\tau}_{\alpha,\alpha}(s)|ds\leq C.
   \qedhere
\]
\end{proof}

\begin{remark}
In \cite[Lemma 2, p.\ 346]{Cermak_2}, the authors also studied the asymptotic behavior of the generalized Mittag-Leffler function $E_{\alpha,\beta}^{a,b,\tau}(t)$ for $\beta=1$ and $\beta=\alpha$.  The key point in the proof of this result is to estimate the quantities $\omega_{(1-\beta)/\alpha+1,3}(t)$ and $\omega_{(1-\beta)/\alpha,3}(t)$, see \cite[l.\ 18, p.\ 347]{Cermak_2}. Those estimates are based on \cite[Proposition 5(ii)]{Cermak_2}. However, there is a gap in the proof of \cite[Proposition 5(ii)]{Cermak_2}. Indeed, they first give the following inequality as $t\to\infty$
\begin{equation}\label{r_1}
|\omega_{x,n}(t)|\leq \frac{C_x}{t^\alpha},
\end{equation}
where $x>-1$, $n=2,3,\dots$, and $$C_x=\frac{1}{2\pi\alpha \eta_0}\left(\alpha(\pi+2\delta)+\frac{\alpha\Gamma(\alpha(x+1))}{(\sin\delta)^{\alpha(x+1)}}\right),$$
see \cite[l.\ 18, p.\ 345]{Cermak_2}. Then, they use the representation
\begin{equation}\label{r_2}
\omega_{x,n+m}(t)=\omega_{x,n}(t)+\sum_{j=1}^{\infty}\frac{(m\tau)^j}{j!t^j}\omega_{j/\alpha+x,n}(t),
\end{equation}
where $m\in \Z^+$ is arbitrary. Finally, they apply \eqref{r_1} for the term $\omega_{j/\alpha+x,n}(t)$ in \eqref{r_2} to show that 
\[
\sum_{j=1}^{\infty}\frac{(m\tau)^j}{j!t^j}\omega_{j/\alpha+x,n}(t)=\mathcal{O}(t^{-\alpha-1})\quad\text{as}\; t\to\infty.
\]
In our opinion, this argument maybe not true due to the fact that the coefficients $C_{j/\alpha+x}$ in the estimate for $\omega_{j/\alpha+x,n}(t)$ (by using \eqref{r_1} as above) are not bounded as $j\to \infty$.
\end{remark}

\section{Asymptotic stability}

Our aim in this section is to prove the following theorem.

\begin{theorem}[Stability of scalar nonlinear fractional differential equation with linearly dominated delay]\label{Main result}
 Let $\tau>0$, $a,b\in \R$ with $a\leq b<-a$ and $f$ satisfy $\textup{(H1)}$ and $\textup{(H2)}$. Then, the trivial solution of 
 the initial value problem \eqref{IntroEq}, \eqref{Ini_Cond}, is asymptotically stable.
\end{theorem}
\begin{proof}
From the assumption $\textup{(H2)}$, we have a constant $\varepsilon_0>0$ such that
\begin{equation*}\label{m_1}
q:=\ell_f(\varepsilon)\; C<1
\end{equation*}
for all $\varepsilon\in (0,\varepsilon_0)$, where $C$ is the constant chosen in Lemma \ref{est_lemma}. Let $\varepsilon>0$ (w.l.o.g.\  $\varepsilon\leq \varepsilon_0$) and choose $\delta>0$ satisfying
\[
\delta=\frac{(1-q)\varepsilon}{\sup_{t\geq 0}|E^{a,b,\tau}_{\alpha,1}(t)|+|b|\int_0^\infty|E^{a,b,\tau}_{\alpha,\alpha}(t)|dt+1}.
\]
Let $F \colon \R^2\to \R$ be a Lipschitz continuous function with Lipschitz constant $\ell_f(\varepsilon)$ and $F(x,y) = f(x,y)$ for all $(x,y)\in \R^2$ such that $\max\{|x|,|y|\}\leq \varepsilon$. Such a Lipschitz extension always exists, see e.g., \cite[Theorem 2.5]{Heinonen}. Consider the equation
\begin{equation}\label{m_2}
\capD^\alpha_{0+}x(t)=ax(t)+bx(t-\tau)+F(x(t),x(t-\tau)), \quad t\geq 0,
\end{equation}
with the initial condition $x(t)=\phi(t)$ for all $t\in [-\tau,0]$, where $\phi\in B_{C([-\tau,0];\R)}(0,\delta)$. From Lemma \ref{var_const_for}, we see that the unique solution $\hat\varphi(\cdot,\phi)$ of \eqref{m_2} has the representation
\begin{align*}
\hat\varphi(t,\phi)&=\phi(0)E^{a,b,\tau}_{\alpha,1}(t)+b\int_{-\tau}^{t-\tau} E^{a,b,\tau}_{\alpha,\alpha}(t-\tau-s)\tilde\phi(s)ds\\
&\hspace{0.5cm}+\int_0^t E^{a,b,\tau}_{\alpha,\alpha}(t-s)F(\hat\varphi(s),\hat\varphi(s-\tau))ds
\quad \text{for all } t\geq 0,
\end{align*}
and $\hat\varphi(t,\phi)=\phi(t)$ on $[-\tau,0]$.

Next, we introduce a Lyapunov--Perron operator on $C([-\tau,\infty);\R)$ as follows. Given any $\phi\in C([-\tau,0];\R)$, the operator $\mathcal{T}_{\phi,\tau}$ on $C([-\tau,\infty);\R)$ is defined by
\begin{align*}
&\mathcal{T}_{\phi,\tau}\xi(t)=\phi(0)E^{a,b,\tau}_{\alpha,1}(t)+b\int_{-\tau}^{t-\tau} E^{a,b,\tau}_{\alpha,\alpha}(t-\tau-s)\tilde\phi(s)d\tau\\
&\hspace{1.5cm}+\int_0^t E^{a,b,\tau}_{\alpha,\alpha}(t-s)F(\xi(s),\xi(s-\tau))ds \quad \text{for all } t\geq 0,\\
&\mathcal{T}_{\phi,\tau}\xi(t)=\phi(t) \quad \text{for all } t\in [-\tau,0].
\end{align*}
For $\phi\in B_{C([-\tau,0];\R)}(0,\delta)$, it is easy to see that for $\xi\in B_{C_{\infty}}(0,\varepsilon)$
\begin{align*}
&\|\mathcal T_{\phi, \tau}\xi\|_\infty \leq \;\Big(\sup_{t\geq 0}|E_{\alpha,1}^{\lambda_i,\tau}(t)|+|b|\int_0^\infty |E_{\alpha,\alpha}^{\lambda_i,\tau}(s)|\;ds+1\Big)\;\|\phi\|_\infty+C\; \ell_{f}(\varepsilon)\;\|\xi\|_{\infty}\\
&\hspace{1.5cm} \leq \;(1-q)\varepsilon+q\varepsilon
\\
&\hspace{1.5cm}=\varepsilon,
\end{align*}
which proves that $\mathcal T_{\phi,\tau}(B_{C_{\infty}}(0,\varepsilon))\subseteq B_{C_{\infty}}(0,\varepsilon)$, and
\begin{eqnarray*}
\|\mathcal T_{\phi,\tau}\xi-\mathcal T_{\phi,\tau}\widehat{\xi}\|_\infty &\leq&
C\; \ell_{f}(\varepsilon)\;\|\xi-\widehat{\xi}\|_\infty\\[1.5ex]
&\leq & q\|\xi-\widehat{\xi}\|_\infty \quad \text{for all } \xi,\hat\xi\in B_{C_{\infty}}(0,\varepsilon).
\end{eqnarray*}
By using the Banach fixed point theorem, we see that there exists a unique fixed point $\xi^*$ of $\mathcal T_{\phi,\tau}$ in $B_{C_{\infty}}(0,\varepsilon)$. The uniqueness of the solution to \eqref{m_2} implies that $\hat\varphi(t,\phi)=\xi^*(t)$ for all $t\in [-\tau,\infty)$. Thus, $\|\hat\varphi(\cdot,\phi)\|_\infty\leq \varepsilon$ and
\begin{align*}
\capD^\alpha_{0+}\hat{\varphi}(t,\phi)&=a\hat{\varphi}(t,\phi)+b\hat{\varphi}(t-\tau,\phi)+F(\hat{\varphi}(t,\phi),\hat{\varphi}(t-\tau,\phi))\\
&=a\hat{\varphi}(t,\phi)+b\hat{\varphi}(t-\tau,\phi)+f(\hat{\varphi}(t,\phi),\hat{\varphi}(t-\tau,\phi)) \quad \text{for all } t\geq 0,
\end{align*}
which implies that the trivial solution to \eqref{IntroEq} is stable. Finally, we will show that the trivial solution to \eqref{IntroEq} is attractive. Suppose that $\xi(t)$ is the solution of \eqref{IntroEq}, \eqref{Ini_Cond} which satisfies $\xi(t)=\phi(t)$ for every $t\in[-\tau,0]$, where $\phi\in B_{C([-\tau,0];\R)}(0,\delta)$. As shown above, we see that $\|\xi\|_\infty\le \varepsilon$. Let
$a \coloneqq \limsup_{t\to\infty}|\xi(t)|$, then $a\in [0,\varepsilon]$. Let $\hat\varepsilon > 0$ small enough. Then, there exists $T(\hat\varepsilon)>0$
such that
\[
|\xi(t)|\le a+\hat\varepsilon \qquad \textup{for all } t\ge T(\hat\varepsilon).
\]
According to Lemma \ref{est_lemma}, we obtain
\begin{itemize}
\item[(i)] $\lim_{t\to \infty}E_{\alpha,1}^{a,b,\tau}(t)=0$,
\item[(ii)] $\lim_{t\to\infty}\int_{-\tau}^{t-\tau}E_{\alpha,\alpha}^{a,b,\tau}(t-\tau-s)\hat{\phi}(s)\;ds=0$,
\item[(iii)]
\begin{eqnarray*}
&&\limsup_{t\to\infty}\left|\int_0^{T(\hat\varepsilon)}E_{\alpha,\alpha}^{a,b,\tau}(t-s)f(\xi(s),\xi(s-\tau))\;ds\right|\\[1.5ex]
&\le& \max_{t\in [0,T(\varepsilon)]}|f(\xi(t),\xi(t-\tau))|\limsup_{t\to \infty}\int_0^{T(\hat\varepsilon)}\frac{C}{(t-s)^{\alpha+1}}ds\\
&= &0.
\end{eqnarray*}
\end{itemize}
Therefore, from the fact that $\xi(t)=(\mathcal{T}_{\phi,\tau} \xi)(t)$, we have
\begin{eqnarray*}
\limsup_{t\to\infty}|\xi(t)| &=&
\limsup_{t\to\infty}\left|\int_{T(\hat\varepsilon)}^tE_{\alpha,\alpha}^{a,b,\tau}(t-s)f(\xi(s),\xi(t-\tau))ds\right|\\
&\le& \ell_f(\varepsilon)\; C\; (a+\hat\varepsilon),
\end{eqnarray*}
where we use the estimate
\begin{eqnarray*}
\left|\int_{T(\hat\varepsilon)}^tE_{\alpha,\alpha}^{a,b,\tau}(t-s)\;ds\right| &=&
\int_0^{\infty}|E_{\alpha,\alpha}^{a,b,\tau}(u)|\;du\\
&\le& C,
\end{eqnarray*}
see Lemma \ref{est_lemma}(iii), to obtain the inequality above. Thus,
\begin{align*}
a &\le \ell_f(\varepsilon) C (a+\hat\varepsilon).
\end{align*}
Letting $\hat\varepsilon\to 0$, we have
\[
a \le \ell_f(\varepsilon) C a.
\]
Due to the fact $\ell_f(\varepsilon) C<1$, we get that $a=0$ and the proof is complete.
\end{proof}
To complete this paper, we give an example to illustrate the main result.

\begin{example}\label{ex1}
The fractional differential equation
\begin{equation}\label{ex}
\capD^{0.5}_{0+}x(t)=-5x(t)+0.5 x(t-1)+x^2(t)+x^3(t-1)
\end{equation}
is of the form \eqref{IntroEq} with $a=-5$, $b=0.5$, $f(x,y)=x^2+y^3$, and satisfies the assumptions of Theorem \ref{Main result}. Its trivial solution is therefore asymptotically stable. 

\begin{figure}[h!]
\centering
\includegraphics[width=1\textwidth]{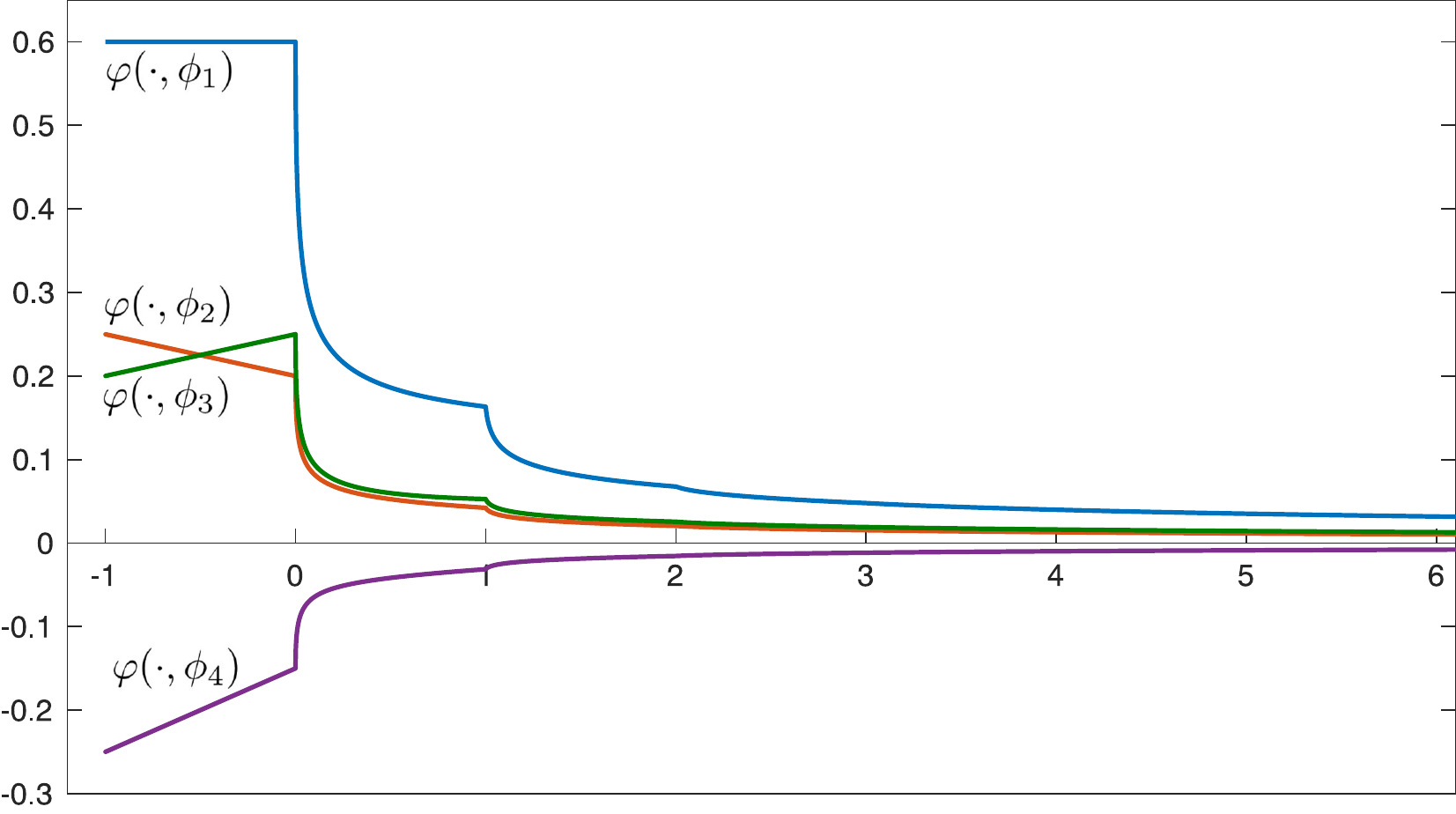}
\caption{The solutions $\varphi(\cdot,\phi_i)$, $i=1,2,3,4$, from Example \ref{ex1}.\label{figsolutions}}
\end{figure}
Using an Adams-Bashforth-Moulton predictor-corrector scheme for fractional differential equations \cite{BhalekarEtal2011, DFF2002},  solutions $\varphi(\cdot, \phi_i)$ to the equation \eqref{ex} are computed for the initial conditions $\varphi(t,\phi_i)=\phi_i(t)$ for $t\in [-1,0]$, $i=1, 2, 3, 4$, with the initial functions $\phi_i \colon [-1,0]\to\R$ defined by
\begin{align*}
   \phi_1(t) &= 0.6,
\\
   \phi_2(t) &= - 0.05 t +0.2,
\\
   \phi_3(t) &= 0.05 t + 0.25,
\\
   \phi_4(t) &= 0.1 t - 0.15,
\end{align*}
see Figure \ref{figsolutions}.
\end{example}
\section*{Acknowledgement}
The research of Hoang The Tuan was supported by the bilateral project between FWO Flanders and NAFOSTED Vietnam (FWO.101.2017.01). This paper was done when he visited the Center for Dynamics at TU Dresden, Germany, with the support of Deutscher Akademischer Austauschdienst (DAAD). The authors thank Ninh Van Thu and Hieu Trinh for helpful discussions.

\end{document}